\documentclass[10pt]{amsart}
\usepackage{amssymb}
\usepackage{color}
\usepackage{graphicx}
\usepackage{float}
\usepackage[all,cmtip]{xy}
\newcommand{\D}{\mathbb{D}}

\newcommand{\s}{\mathfrak{s}}

\newcommand{\SU}{{\mathcal{U}}}

\newcommand{\Z}{\mathbb{Z}}
\newcommand{\C}{\mathbb{C}}

\newcommand{\R}{\mathbb{R}}

\renewcommand{\S}{\mathbb{S}}

\newtheorem{proposition}{Proposition}
\newtheorem{theorem}[proposition]{Theorem}
\newtheorem{definition}[proposition]{Definition}
\newtheorem{lemma}[proposition]{Lemma}

\newtheorem{corollary}[proposition]{Corollary}

\begin{document}

\title[Small positive loops of contactomorphisms]{On the non--existence of small positive loops of contactomorphisms on overtwisted contact manifolds}

\date{November, 2013}


\author{Roger Casals}
\address{Instituto de Ciencias Matem\'aticas -- CSIC.
C. Nicol\'as Cabrera, 13--15, 28049, Madrid, Spain.}
\email{casals.roger@icmat.es}

\author{Francisco Presas}
\address{Instituto de Ciencias Matem\'aticas -- CSIC.
C. Nicol\'as Cabrera, 13--15, 28049, Madrid, Spain.}
\email{fpresas@icmat.es}

\author{Sheila Sandon}
\address{Universit\'{e} de Strasbourg and CNRS, 67084 Strasbourg, France}
\email{sandon@math.unistra.fr}



\begin{abstract}
We prove that on overtwisted contact manifolds there can be no positive loops of contactomorphisms that are generated by a $\mathcal{C}^0$--small Hamiltonian function.
\end{abstract}

\maketitle


\section{Introduction}\label{sec: intro}

In 2000 Eliashberg and Polterovich \cite{EP} noticed that the natural notion of positive contact isotopies, i.e. contact isotopies that move every point in a direction positively trasverse to the contact distri\-bution, induces for certain contact manifolds a partial order on the universal cover of the contactomorphism group. Such contact manifolds are called \textit{orderable}. Since the work of Eliashberg and Polterovich orderability has become an important subject in the study of contact topology. In particular it has been discovered to be deeply related to the contact non--squeezing phenomenon \cite{EKP,Gi} and, more recently, to the non--degeneracy of a natural bi--invariant metric that is defined on the universal cover of the contactomorphism group \cite{CS}. 

As Eliashberg and Polterovich explained, orderability of a contact manifold is equivalent to the non--existence of a positive contractible loop of contactomorphisms. By now many contact manifolds are known to be orderable and many are known not to be, but it is still not well--understood where the boundary between the orderable and non--orderable world lies. In particular it is not known whether there is a relation between overtwistedness and orderability, since not a single overtwisted contact manifold is known to be orderable or not to be. In this article we prove the following result.

\begin{theorem}\label{thm:otnotsmall}
Let $\big(M,\xi = \ker\alpha\big)$ be a closed overtwisted contact 3--manifold. Then there exists a real positive constant $C(\alpha)$ such that any positive loop $\{\phi_{\theta}\}$ of contactomorphisms which is generated by a contact Hamiltonian  $H: M \times\S^1 \longrightarrow \mathbb{R}^+$ satisfies 
$$\|H\|_{\mathcal{C}^0} \geq C(\alpha)\,.$$
\end{theorem}

In other words, on closed overtwisted contact 3--manifolds there are no positive loops of contactomorphisms that are generated by a $\mathcal{C}^0$--small contact Hamiltonian. Note that there is no loss of generality in assuming the contact Hamiltonian to be 1-periodic, see Lemma 3.1.A in \cite{EP}. It is important to notice that our result does not imply that overtwisted contact manifolds are orderable, because the contraction of a positive contractible loop of contactomorphisms is not necessarily performed via positive loops. For instance, it was even proved in \cite{EKP} that for the standard tight contact sphere any contraction of a positive contractible loop must be sufficiently negative somewhere. Theorem \ref{thm:otnotsmall} states though that there exists a lower bound for a Hamiltonian function that generates a positive loop of contactomorphisms. Intuitively, in the presence of an overtwisted disc a positive isotopy returning to the identity requires a minimal amount of energy.

The specificity of our result is that we deal with \textit{$\mathcal{C}^0$--small} contact Hamiltonians. Indeed, let us prove that the non--existence of a positive loop of contactomorphisms that is generated by a \textit{$\mathcal{C}^1$--small} Hamiltonian holds on any contact manifold. Consider first the $\mathcal{C}^2$--small case. If the Hamiltonian $H_{\theta}: M \longrightarrow \mathbb{R}$ is $\mathcal{C}^2$--small then the generated loop $\{\phi_{\theta}\}$ is $\mathcal{C}^1$--small and so the contact graphs\footnote{See for example \cite{S,CS} for the definition of contact graphs, contact products and more details on arguments similar to the one that follows.} $\text{gr}(\phi_{\theta})$ are Legendrian sections in a Weinstein neighborhood of the diagonal $\Delta$ in the contact product $M \times M \times \mathbb{R}$. Since a Weinstein neighborhood is contactomorphic to a neighborhood of the zero section of $J^1(\Delta)=J^1(M)$, the graphs $\text{gr}(\phi_{\theta})$ are of the form $\{j^1f_{\theta}\}$ for a family of smooth functions $f_{\theta}$ on $M$. Because of the Hamilton--Jacobi equation (see \cite[Section 46]{Ar}), positivity of the loop $\{\phi_{\theta}\}$ implies that the family $f_{\theta}$ is strictly increasing, yielding a contradiction. If the Hamiltonian function is only $\mathcal{C}^1$--small, and thus the loop $\{\phi_{\theta}\}$ is $\mathcal{C}^0$--small, then the graphs $\text{gr}(\phi_{\theta})$ are still contained in a Weinstein neighborhood of the diagonal in the contact product but they are not necessarily sections anymore, and so they cannot be written as 1-jet of functions. However it follows from Chekanov theorem \cite{Che,Ch} that they have generating functions quadratic at infinity and so an argument similar to the one above (or the results in \cite{CFP,CN}) allows to conclude also in this case. As far as we know, Theorem \ref{thm:otnotsmall} is the first result in the literature that shows the non--existence of a positive loop in the case when the Hamiltonian is $\mathcal{C}^0$--small. Our proof strongly uses overtwistedness in several points, and does not give an intuition of whether or not the result should also be true for tight contact manifolds. However it seems plausible to us that this might be the case.

Although Theorem \ref{thm:otnotsmall} only applies to overtwisted contact $3$--manifolds, a higher--dimensional analogue can also be stated. The careful reader can try to generalize the result to non--fillable contact manifolds containing a PS--structure \cite{Ni,Pr}, a GPS--structure \cite{NP} or a blob \cite{MNW} with the appropriate hypotheses on the Chern class of the contact distributions. The precise statement is not part of this article due to its technicality and to the fact that no new geometric ideas are required for the argument.

As a consequence of Theorem \ref{thm:otnotsmall}, we can bound from below not only the supremum norm of the Hamiltonian of a positive loop but also its $L^1$-norm, in the following sense.

\begin{corollary}\label{cor:otnotsmall}
Let $\big(M,\xi=\ker\alpha\big)$ be a closed overtwisted contact 3--manifold. Then there exists a real positive constant $C(\alpha)$ such that any positive loop of contactomorphisms $\{\phi_\theta\}$ which is generated by a contact Hamiltonian $H_{\theta}$, $\theta\in\S^1$, satisfies
$$\int_0^1 \| H_{\theta}\|_{\mathcal{C}^0}\,d\theta \geq C(\alpha)\,.$$
\end{corollary}

Corollary \ref{cor:otnotsmall} can be deduced from Theorem \ref{thm:otnotsmall} as follows. Suppose that there is a positive loop $\{\phi_\theta\}$ which is generated by a contact Hamiltonian $H_{\theta}$ that satisfies
$$\int_0^1 \| H_{\theta}\|_{\mathcal{C}^0}\,d\theta \leq C(\alpha)\,.$$
Define a reparametrization $\beta: [0,1] \rightarrow [0,1]$ of the time--coordinate by requiring
$
\dot{\beta}(\theta) = \frac{\| H_{\theta}\|_{\mathcal{C}^0}}{\int_0^1 \| H_{\theta}\|_{\mathcal{C}^0}\,d\theta}
$
and write $\phi_{\theta} = \psi_{\beta(\theta)}$. Then $H_{\theta} = \dot{\beta}(\theta) G_{\beta(\theta)}$ where $G_{\beta(\theta)}$ is the Hamiltonian of the reparametrized loop $\psi_{\beta(\theta)}$. For all $\theta \in \S^1$ we then have 
$$
\max_{x} G_{\beta(\theta)}(x) = \frac{\int_0^1 \| H_{\theta}\|_{\mathcal{C}^0}\,d\theta}{\| H_{\theta}\|_{\mathcal{C}^0}}\| H_{\theta}\|_{\mathcal{C}^0} = \int_0^1 \| H_{\theta}\|_{\mathcal{C}^0}\,d\theta \leq C_{\alpha}
$$
contradicting Theorem \ref{thm:otnotsmall}.

The geometric core of the proof of Theorem \ref{thm:otnotsmall} can be shortly described in two parts. First, any overtwisted contact manifold $(M,\xi)$ can be embedded with trivial symplectic normal bundle in an exact symplectically fillable contact 5--manifold $(X,\xi_X)$. Second, the existence of a small positive loop of contactomorphisms on $(M,\xi)$ implies the existence of a PS--structure on $(X,\xi_X)$. This yields a contradiction, according to the main result of \cite{Ni}. The construction of a PS--structure on $X$ is based on techniques similar to those used by Niederkr\"{u}ger and the second author \cite{NP} to study the size of tubular neighborhoods of contact submanifolds.

The paper is organized as follows. Section \ref{sec: prelim} recalls basic definitions and facts about overtwisted contact manifolds. In Section \ref{sec: contact fibrations} we explain how to construct a PS--structure in the total space of the contact fibration $M \times \mathbb{D}^2$, where $M$ is an overtwisted contact 3--manifold, starting from a small positive loop of contactomorphisms of $M$. Theorem \ref{thm:otnotsmall} is proved in Section \ref{sec: proof} assuming an embedding result that will be proved in Section \ref{sec: embedding}. 

{\bf Acknowledgments.} This work started five years ago as an attempt to prove that overtwisted contact manifolds are orderable. Many visits at that time of the third author to Madrid were supported by the CAMGSD of the IST Lisbon and by the CSIC-IST joint project 2007PT0014. More recently, we wish to thank the participants of the AIM Workshop \textit{Contact topology in higher dimensions} in May 2012 for discussions that encouraged us to pursue and write down this partial result. The article was written during the stay of the third author at the UMI--CNRS of the CRM Montr\'{e}al, and she would like to thank Laurent Habsieger, Fran\c{c}ois Lalonde and Octav Cornea for their support and hospitality. The first author is grateful to A. Zamorzaev. The present work is part of the authors activities within CAST, a Research Network Program of the European Science Foundation. The third author is also supported by the ANR grant COSPIN.

\section{Preliminaries on overtwisted contact manifolds}\label{sec: prelim}

We refer to the book of Geiges \cite{Ge} for an introduction to Contact Topology, and recall here only the definitions and facts about overtwisted contact manifolds that will be needed in the rest of the article. A 3--dimensional contact manifold $(M,\xi)$ is said to be \textit{overtwisted} if it contains an overtwisted disc, i.e. an embedded 2--disk $\Delta$ such that the characteristic foliation $T\Delta\cap\xi$ contains a unique singular point in the interior of $\Delta$ and $\partial\Delta$ is the only closed leaf of this foliation. A contact manifold is said to be tight if it is not overtwisted. 

As follows from the results of Lutz and Martinet \cite{Lu, Ma}, there exists an overtwisted contact structure in any homotopy class of $2$--plane fields. Moreover, by the classification of overtwisted contact structures achieved by Eliashberg \cite{E2}, we also know that on a given homotopy class of $2$--plane fields there exists exactly one overtwisted contact structure. More precisely we have the following result.

\begin{theorem}[\cite{E2}]\label{thm:classification of ot}
Let $\xi$ and $\xi'$ be overtwisted contact structures on a 3--dimensional manifold $M$, and suppose that they are homotopic as $2$--plane fields. Then $\xi$ and $\xi'$ are isotopic contact structures.
\end{theorem}

The notion of an overtwisted contact structure does not readily generalize to higher--dimensional contact manifolds. The following geometric model was proposed by Niederkr\"{u}ger \cite{Ni}. 

\begin{definition}
Let $(M,\xi)$ be a contact $5$--manifold. A \emph{plastikstufe} PS$(\S^1)$ in $M$ with singular set $\S^1$ is an embedding of a solid torus
$$\iota:\D^2\times\S^1\longrightarrow M$$
with the following properties:
\begin{itemize}
\item[a.] The boundary $\partial\D^2\times\S^1$ is the unique closed leaf of the foliation $\ker(\iota^*\alpha)$ on $\D^2\times\S^1$.
\item[b.] The interior of $\D^2\times\S^1$ is foliated by an $\S^1$--family of stripes $(0,1)\times\S^1$ spanned between $\S^1\times\{0\}$ and asymptotically approaching $\partial\D^2\times\S^1$ on the other side.
\end{itemize}
\end{definition}

In particular, Property a. implies that the boundary of the solid torus is a Legendrian torus and the core $\{0\}\times\S^1$ is transverse to the contact distribution $\xi$.

\begin{figure}[H]
\includegraphics[scale=0.55]{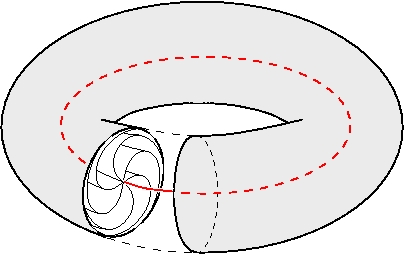} \\
\caption{{\small An embedded PS--structure in a contact 5--fold.}} \label{fig:plastik}
\end{figure}

A plastikstufe is also referred to in the literature as a PS--structure.

By results of Gromov and Eliashberg \cite{Gr,E3}, in dimension $3$ the presence of an overtwisted disc obstructs the existence of symplectic fillings. The higher--dimensional analogue of this fact is the following theorem by Niederkr\"{u}ger.

\begin{theorem}[\cite{Ni}]\label{thm:psnotfill}
Let $(M,\xi)$ be a contact $5$--manifold with a PS--structure. Then $M$ does not admit an exact symplectic filling.
\end{theorem}

As we will explain, the argument used to prove Theorem \ref{thm:otnotsmall} is based on the insertion of a PS--structure in an exact symplectically fillable manifold, thus yielding a contradiction with Theorem \ref{thm:psnotfill}. The techniques that provide such embedding are based on the study of certain contact structures on the manifold $M\times\D^2$. This will be explained in the next section.

\section{PS--structures and contact fibrations}\label{sec: contact fibrations}

In the first part of this section we will recall, following Lerman \cite{Le} and \cite{Pr}, the notion of a contact fibration and its relation to the group of contactomorphisms via the monodromy diffeomorphism. We will then show in Proposition \ref{prop: construction of PS} how to apply these concepts to construct a PS--structure on the total space of the contact fibration $M \times \mathbb{D}^2$, starting from a sufficiently small positive loop of contactomorphisms of $M$.

A smooth fiber bundle $\pi:X\longrightarrow B$ is said to be a \textit{contact fibration} if there exists a hyperplane distribution $\xi_X = \ker\alpha_X$ on $X$ such that its restriction $\xi = \ker (\pi) \cap \xi_X$ defines a contact structure in each fiber. In particular $\big(\xi,d \alpha_X|_{\ker(\pi)}\big)$ is a symplectic subbundle of the --not necessarily symplectic-- bundle $\xi_X$. This data leads to a natural choice of connection.

\begin{definition}
Let $\pi:\big(X,\xi_X = \ker\alpha_X\big)\longrightarrow B$ be a contact fibration. Then the distribution $\xi^{\perp d\alpha_X}\subset \xi_X$ is called the \emph{contact connection} associated to the contact fibration.
\end{definition}

In other words, for a point $p$ of $B$ and a tangent vector $v \in T_pB$, the horizontal lift of $v$ at some $\tilde{p} \in \pi^{-1}(p)$ with respect to the contact connection is the unique vector $\tilde{v} \in T_{\tilde{p}}X$ such that $\pi_{\ast} \tilde{v} = v$, $\tilde{v} \in \ker (\xi_X)$ and $\iota_{\tilde{v}}d\alpha_X = 0$ on $\xi$. Note that the contact connection only depends on $\xi_X$, not on the choice of the 1--form $\alpha_X$ with $\xi_X = \ker\alpha_X$. The parallel transport along a segment joining two points $q,p\in B$ is defined as in the smooth case, but in the contact framework it is enhanced from a diffeomorphism to a contactomorphism between the fibers of $q$ and $p$. Moreover, the definition of the contact connection implies that the trace by parallel transport of a submanifold that is tangent to the contact structure on the fibers is also tangent to the distribution on the total space. A precise statement of these properties is the content of the following proposition.

\begin{proposition}\label{pro:contcon}$($\cite{Le,Pr}$)$
Let $\pi:\big(X,\xi_X=\ker\alpha_X\big)\longrightarrow B$ be a contact fibration with closed fibers. Consider a point $p\in B$ and an immersed path $\gamma:[0,1]\longrightarrow B$ with $\gamma(0) = p$. Then parallel transport along $\gamma$ with respect to the contact connection defines a path of diffeomorphisms
$$\widetilde{\gamma}_t:\pi^{-1}(p)\longrightarrow\pi^{-1}(\gamma(t))$$
with the following properties:
\begin{itemize}
\item[a.] The diffeomorphisms $\widetilde{\gamma}_t$ are contactomorphisms.
\item[b.] Let L be an isotropic submanifold of $\pi^{-1}(p)$ and consider the map
$$\mathfrak{t}:L\times[0,1]\longrightarrow X,\quad (p,t)\longmapsto \widetilde{\gamma}_t(p),$$
then $im(\mathfrak{t})$ is an immersed isotropic submanifold of $(X,\xi_X)$. It is an embedded isotropic submanifold if $\gamma$ is an embedded path.
\end{itemize}
\end{proposition}

Note that the closedness condition for the fibers is technical and only used to ensure that the vector fields implicitly appearing in the statement are complete.

There are instances in which the contactomorphisms generated via parallel transport have a simple description. The following example will be used in the proof of our results.

Let $\big(M,\xi=\ker\alpha\big)$ be a contact manifold. A time--dependent function $H_{\theta}$ on $M$ induces a path of contactomorphisms $\{\phi_{\theta}\}$, which is defined to be the flow of the time--dependent vector field $X_{\theta}$ satisfying
\begin{eqnarray}\label{eq:contHam}
\iota_{X_{\theta}} \alpha & = & H_{\theta}, \\
\iota_{X_{\theta}} d\alpha &= & -dH_{\theta} + dH_{\theta}(R_{\alpha})\,\alpha \nonumber
\end{eqnarray}
where $R_{\alpha}$ is the Reeb vector field associated to $\alpha$. The function $H_{\theta}$ is called the \textit{contact Hamiltonian} with respect to the contact form $\alpha$ of the contact isotopy $\{\phi_{\theta}\}$. In contrast to the symplectic case, any contact isotopy can be written as the flow of a contact Hamiltonian, see \cite[Section 2.3]{Ge}.

Consider the manifold $M \times \D^2$, where $\D^2$ denotes the 2--disc with polar coordinates $(r,\theta)$. Let $H:M\times\D^2 \longrightarrow\R$ be a function such that $H\in O(r^2)$ at the origin and $\partial_r H>0$. Then the $1$--form
$$\alpha_H=\alpha + H(p,r,\theta)d\theta$$
defines a contact structure $\xi_H$ on the manifold $M\times\D^2$. In particular, suppose that $H:M\times\S^1\longrightarrow\R$ is a positive function. Then $\alpha_H = \alpha + H(p,\theta)\cdot r^2d\theta$ is a contact form in $M\times\D^2$. 

\begin{lemma}\label{lem:monHam}
Let $\big(M,\xi=\ker\alpha\big)$ be a contact manifold, and $H_\theta:M\longrightarrow\R$ an $\S^1$--family of positive smooth functions. Consider the contact fibration
$$\pi:\big(M\times\D^2,\ker\alpha_H\big)\longrightarrow \D^2.$$
Then parallel transport along $\gamma(\theta)=(1,-\theta)$ is the contact flow of the Hamiltonian $H_\theta$.
\end{lemma}

\begin{proof}
The horizontal lift with respect to the contact connection of the vector field $\partial_\theta$ at a point $(1,\theta)$ is of the form
$$\widetilde{X}=\partial_\theta-X_\theta,$$
where $X_\theta$ satisfies the equations $\iota_{X_{\theta}} \alpha = H_{\theta}$ and $\iota_{X_{\theta}} d\alpha = -dH_{\theta} + dH_{\theta}(R_{\alpha})\,\alpha$. Indeed, the lift is unique and $\widetilde{X}$ satisfies both $\alpha_H(\widetilde{X})=0$ and $\iota_{\widetilde{X}}d\alpha_H = 0$ on $\xi$. The statement then follows from equations (\ref{eq:contHam}).
\end{proof}

Let us explain how to use Lemma \ref{lem:monHam} to construct a PS--structure in $\big(M \times \D^2(\delta),\ker(\alpha+r^2d\theta)\big)$, where $\D^2(\delta)$ denotes the 2--disc of radius $\delta$, assuming that there is a sufficiently small positive loop of contactomorphisms in $M$.

\begin{proposition}\label{prop: construction of PS}
Assume that $\{\phi_t\}$ is a positive loop of contactomorphisms of an overtwisted contact manifold $\big(M,\xi=\ker\alpha\big)$ which is generated by a contact Hamiltonian $H_{\theta}$, $\theta \in S^1$, with $H_{\theta} < \delta^2$ for some $\delta \in \mathbb{R}^+$. Then there is a PS--structure on $\big(M \times \D^2(\delta),\ker(\alpha+r^2d\theta)\big)$.
\end{proposition}

\begin{proof}
Note first the following general fact. Suppose that that $\pi:(X,\xi_X) \longrightarrow \Sigma$ is a contact fibration over a smooth compact surface $\Sigma$, such that the fibers are closed overtwisted contact manifolds. Suppose also that there exists an embedded loop $\gamma: \S^1 \longrightarrow \Sigma$ whose time--1 parallel transport $\widetilde{\gamma}_1$ is the identity. Then there exists a PS--structure in the pre--image $\pi^{-1}(\gamma(\S^1))$. Indeed, since the fiber $\pi^{-1}(\gamma(0))$ is overtwisted we can consider an embedded overtwisted disk $\Delta$ in it and define the map
\begin{eqnarray*}
\rho:  \Delta \times \S^1 & \longrightarrow & X \\
 (p, \theta) &\longmapsto & \rho(r,\theta)=\widetilde \gamma_\theta (p).
\end{eqnarray*}
Then property b. in Proposition \ref{pro:contcon} implies that $im(\rho)$ is a PS--structure. By combining this fact with Lemma \ref{lem:monHam} we see that if $\{\phi_{\theta}\}$ is a positive loop of contactomorphisms on a contact manifold $\big(M,\xi = \ker\alpha\big)$ then there is a PS--structure on $\big(M\times \mathbb{D}^2(1),\text{ker}(\alpha+H_{\theta}r^2d\theta)\big)$ where $H_{\theta}$ is the Hamiltonian function of $\{\phi_{\theta}\}$. The PS--structure is at the level $\{r=1\}$. Note that if $H_\theta < \delta^2$ for some $\delta\in\R^+$ then there exists a strict contact embedding
$$\big(M\times \mathbb{D}^2(1),\text{ker}(\alpha+H_{\theta}r^2d\theta)\big) \longrightarrow \big(M\times \mathbb{D}^2(\delta),\text{ker}(\alpha+r^2d\theta)\big)$$ given by the map $(p, r, \theta) \mapsto (p, \sqrt{H_\theta(p)}r, \theta)$. A PS--structure in 
$\big(M\times \mathbb{D}^2(1),\text{ker}(\alpha+H_{\theta}r^2d\theta)\big)$ at the level $\{r=1\}$ is sent to a PS--structure in $\big(M\times \mathbb{D}^2(\delta),\text{ker}(\alpha+r^2d\theta)\big)$ at the level $\{r = \sqrt{H_{\theta}}\}$. We have thus obtained the required PS--structure in $\big(M \times \D^2(\delta),\ker(\alpha+r^2d\theta)\big)$.
\end{proof}

\section{Proof of Theorem \ref{thm:otnotsmall}}\label{sec: proof}

In this section we prove Theorem \ref{thm:otnotsmall} in the case when $c_1(\xi) = 0$. As we will explain, the general case also follows from the same argument modulo Proposition \ref{prop:symptrivnormal} that will be proved in the last section. 

Let $(M,\xi)$ be a 3--dimensional overtwisted contact manifold and assume that $\{\phi_{\theta}\}$ is a positive loop of contactomorphisms, generated by a contact Hamiltonian $H_{\theta}$, $\theta \in S^1$. We want to show that if $H_{\theta}$ is small in the $\mathcal{C}^0$--norm then the existence of $\{\phi_{\theta}\}$ gives a contradiction with Theorem \ref{thm:psnotfill}.

Recall from the previous section that if $\{\phi_t\}$ is a positive loop of contactomorphisms of $M$ which is generated by a sufficiently small contact Hamiltonian $H_{\theta}$ ($\theta \in S^1$) then there is a PS--structure on $\big(M \times \D^2(\delta),\ker(\alpha+r^2d\theta)\big)$ for some small $\delta \in \mathbb{R}^+$. Note that the manifold $\big(M\times \mathbb{D}^2(\delta),\text{ker}(\alpha+r^2d\theta)\big)$ is the standard contact neighborhood of a codimension--2 contact submanifold with trivial symplectic normal bundle, see \cite[Section 2.5.3]{Ge}. The result of the previous section implies thus the following proposition.

\begin{proposition}\label{prop:pslocal}
Let $(X,\xi_X)$ be a contact 5--manifold and $(M,\xi)$ a codimension--2 overtwisted contact 3--manifold with trivial symplectic normal bundle. Suppose that $\{\phi_\theta\}$ is a positive loop of contactomorphisms which is generated by a sufficiently small contact Hamiltonian. Then there exists a PS--structure in a neighborhood of $M$ in $X$.
\end{proposition}

In the case when $c_1(\xi) = 0$, Theorem \ref{thm:otnotsmall} follows from Proposition \ref{prop:pslocal}. Indeed any contact manifold $(M,\xi)$ can be embedded as a contact submanifold into its unit cotangent bundle $\S T^{\ast}M$, for example by the map $e_{\alpha}: M \longrightarrow \S T^{\ast}M$ defined by $e_{\alpha}(p) = \big(p,\alpha(p)\big)$ where $\alpha$ is a contact form for $\xi$. Note that if $c_1(\xi) = 0$ then $e_{\alpha}(M) \subset \S T^{\ast}M$ has trivial symplectic normal bundle. Certainly, the symplectic normal bundle of $e_{\alpha}(M)$ inside $\S T^{\ast}M$ is isomorphic to $\xi$, and thus it is trivial if its Euler class $c_1(\xi)$ vanishes. If there was a small positive contact Hamiltonian $H_{\theta}$ that generates a loop of contactomorphisms then Proposition \ref{prop:pslocal} would give a PS--structure inside $\S T^{\ast}M$. But the existence of a PS--structure inside $\S T^{\ast}M$ is impossible by Theorem \ref{thm:psnotfill} because $\S T^{\ast}M$ is an exact symplectially fillable manifold, a filling being given by $\D T^{\ast}M$. More precisely, a tubular neighborhood of $e_{\alpha}(M)$ inside $\S T^{\ast}M$ is contactomorphic to $\big(M\times \mathbb{D}^2(\delta), \text{ker}(\alpha + r^2d\theta)\big)$ for some $\delta > 0$. If the $\mathcal{C}^0$--norm of $H_{\theta}$ is smaller than $\delta^2$ then we would obtain a PS--structure inside $\S T^{\ast}M$. The square of the maximal size $\delta$ of a tubular neighborhood $M\times \mathbb{D}^2(\delta)$ of $M$ inside $\S T^{\ast}M$ gives in this case the constant $C(\alpha)$ that appears in the statement of Theorem \ref{thm:otnotsmall}.

In the general case, i.e. when $c_1(\xi)$ does not necessarily vanish, the proof of Theorem \ref{thm:otnotsmall} follows from the same argument, combined with the following proposition.

\begin{proposition}\label{prop:symptrivnormal}
Every 3--dimensional overtwisted contact manifold can be embedded as a contact submanifold with trivial symplectic normal bundle into an exact symplectically fillable contact 5--manifold.
\end{proposition}

The proof of this result will be given in the next section. Assuming it, Theorem \ref{thm:otnotsmall} is proved as follows. Given an overtwisted contact 3--manifold $(M,\xi)$, by Proposition \ref{prop:symptrivnormal} it can be embedded with trivial symplectic normal bundle into an exact symplectically fillable contact 5--manifold $(X,\xi_X)$. If there was a sufficiently small contact Hamiltonian $H_{\theta}$ generating a positive loop of contactomorphisms then Proposition \ref{prop:pslocal} would give a PS--structure in $X$, contradicting Theorem \ref{thm:psnotfill}.

\section{Contact embeddings with trivial normal bundle}\label{sec: embedding}

In this section we will prove Proposition \ref{prop:symptrivnormal}, i.e. that every overtwisted contact 3--manifold $(M,\xi)$ can be embedded with trivial symplectic normal bundle into an exact symplectially fillable contact 5--manifold $(X,\xi_X)$. The idea of the proof is to start with a contact embedding of $(M,\xi)$ into its unit cotangent bundle $\S T^{\ast}M$ and then perform contact surgeries in an appropriate way in order to make the symplectic normal bundle trivial while keeping the symplectic fillability of the resulting 5--manifold. As we will see the process will also modify the contact structure on the initial overtwisted 3--manifold $M$. One of the crucial points of the proof will be to make sure that the modified contact structure on the 3--manifold will still be overtwisted and moreover in the same homotopy class as cooriented 2--plane fields as the initial one. Then it will be isotopic to it according to Theorem \ref{thm:classification of ot}.

We start by briefly recalling the notion of Lutz twist, its effect on the homotopy class of the contact structure and its relation to contact surgery and symplectic cobordism. See \cite{Ge} for more details on these notions.

Let $K$ be a positive transverse knot in $(M,\xi)$. A \textit{Lutz twist} along $K$ is an operation that deforms in a certain way (see \cite[Section 4.3]{Ge}) the contact structure in a neighborhood of $K$. The resulting contact structure $\xi^K$ on $M$ is always overtwisted. The effect of a Lutz twist on the homotopy class of the contact structure can be described as follows. Given two 2--plane fields $\xi_0$ and $\xi_1$ there are two cohomology classes $d^2(\xi_0,\xi_1)\in H^2(M,\Z)$ and $d^3(\xi_0,\xi_1)\in H^3(M,\Z)$ that measure the obstruction for $\xi_0$ and $\xi_1$ to belong to the same homotopy class of plane fields. We refer to \cite{Ge} for details and for a proof of the following results.

\begin{proposition}\label{prop:d2}
Let $K\subset M$ be a positive transverse knot on $\xi$. Then $d^2(\xi,\xi^K) = -pd([K])$.
\end{proposition}

\begin{proposition}\label{prop:d3}
Let $K\subset M$ be a null--homologous positive transverse knot on $\xi$ with self--linking number $sl(K)$. Then $d^2(\xi,\xi^K)=0$ and $d^3(\xi,\xi^K)=sl(K)$.
\end{proposition}

Following Eliashberg \cite{E90} and Weinstein \cite{W}, fix a Legendrian knot $L$ on a contact manifold $3$--manifold $M$ and fix the relative ($-1$)--framing with respect to the canonical contact framing associated to the knot. If we perform on $M$ a handle attachment along $L$, then the resulting cobordism has a natural symplectic structure. The bottom boundary of this cobordism, i.e. the initial contact manifold, is a concave boundary of the symplectic structure. The upper boundary is convex and therefore it has an induced contact structure, which is said to be obtained from the initial one by contact $(-1)$--surgery. The inverse operation is called a contact $(+1)$--surgery.

As proved by Ding, Geiges and Stipsicz \cite{DGS}, the effect of a Lutz twist on a contact manifold can be described in terms of contact surgery as follows. Given a Legendrian knot $L\subset(M,\xi)$, denote by $t(L)$ a positive transverse push--off of $L$ and by $\sigma(L)$ a Legendrian push--off of $L$ with two added zig--zags. Then we have the following result.

\begin{proposition}[\cite{DGS}]\label{prop:dgs}
Let $(M,\xi)$ be a contact $3$--manifold and $L$ a Legendrian knot for $\xi$. The contact structure obtained by a Lutz twist along $t(L)$ is isotopic to the contact structure resulting from a contact $(+1)$--surgery along $L$ and $\sigma(L)$.
\end{proposition}

Being the inverse of a ($-1$)--surgery, the contact ($+1$)--surgery in a contact 3--fold corresponds to a symplectic 2--handle attachment to the concave boundary of a bounded part of the symplectization, i.e. we obtain a symplectic cobordism in which the new boundary is concave. Consider the transverse knot $K=t(L)\subset (M,\xi^K)$ and the belt spheres $\lambda_K,\lambda^\sigma_K\subset(M,\xi^K)$ corresponding to the contact $(+1)$--surgeries along $L$ and $\sigma(L)$ in $(M,\xi)$ described in Proposition \ref{prop:dgs}. Then $\lambda_K$ and $\lambda^\sigma_K$ are two Legendrian knots in $(M,\xi^K)$. Since Proposition \ref{prop:dgs} is a local result, both Legendrian knots can be assumed arbitrarily close to $K$. The following observation will be used in our argument.

\begin{lemma}\label{lem:h}
$[K]=[L]=[\lambda_K]$.
\end{lemma}

\begin{proof}
By definition $[K]=[t(L)]=[L]$. The equality $[L]=[\lambda_K]$ follows from the fact that the surgery in Proposition \ref{prop:dgs} is smoothly trivial. This implies the statement. See Proposition 6.4.5 in \cite{Ge} for further details.
\end{proof}

\noindent A consequence of the description in Proposition \ref{prop:dgs} is the existence of an exact symplectic cobordism realizing a Lutz twist. More precisely we have the following result.

\begin{corollary} \label{coro:dgs}
Let $(M,\xi)$ be a contact 3--manifold and $K$ a positive transverse knot. Then there exists an exact symplectic cobordism $(W,\omega)$ from $(M,\xi^K)$ to $(M,\xi)$, which is realized by a 2--handle attachment along the Legendrian link $\lambda_K\cup\lambda^\sigma_K$.
\end{corollary}

The convex end of $(W,\omega)$ is the contact boundary $(M,\xi)$, the concave end is $(M,\xi^K)$. A Lutz untwist is thus tantamount to an exact symplectic cobordism. It is central to note that the convex end of an exact symplectic cobordism is exact symplectically fillable if the concave end is. This fact will be crucial in our proof of Proposition \ref{prop:symptrivnormal}, because it will ensure that the 5--manifold $X$ into which we will embed $(M,\xi)$ will still be fillable. Indeed, as we will see, $X$ will be obtained by constructing an exact symplectic cobordism between contact 5--manifolds with an exact symplectically fillable concave end. This cobordism will restrict to a cobordism between contact 3--manifolds as the one described in Corollary \ref{coro:dgs}.

In our argument we will also use the following result.

\begin{lemma}\label{lem:formLutz}
Let $(M,\xi)$ be an overtwisted contact 3--manifold and $\Delta$ a fixed overtwisted disk. Consider a Legendrian link $L$ in $M$ disjoint from $\Delta$. Then there exists a Legendrian link $\Lambda$ disjoint from $L\cup\Delta$ such that $\xi$ is isotopic to $\xi^{t(L\cup\Lambda)}$.
\end{lemma}

\begin{proof}
Consider a Legendrian link $\widetilde L$ disjoint from $L\cup\Delta$ and with homology class $[\widetilde L]=-[L]$. Then Proposition \ref{prop:d2} implies that $d^2(\xi,\xi^{t(L\cup\widetilde L)})=0$. Let $K$ be a null--homologous knot contained in a Darboux ball with self--linking number $-d^3(\xi,\xi^{t(L\cup\widetilde L)})$. Propositions \ref{prop:d2} and \ref{prop:d3} imply that $\Lambda=\widetilde L\cup K$ satisfies $d^2(\xi,\xi^{t(L\cup\Lambda)})=0$ and $d^3(\xi,\xi^{t(L\cup\Lambda)})=0$. Theorem \ref{thm:classification of ot} concludes the statement of the Lemma.
\end{proof}

We are now almost ready to state and prove two results, Propositions \ref{prop:trivnorm} and \ref{prop:trivnormK}, that will be the two main steps in the proof of Proposition \ref{prop:symptrivnormal}. Proposition \ref{prop:trivnorm} will be an adaptation to higher dimensions of Proposition \ref{prop:dgs}. We first discuss the smooth model for it.

Denote by $M_K(\tau)$ the manifold obtained by surgery along $K\subset M$ with framing $\tau$. In case this is a contact surgery along a Legendrian knot, the notation stands for a contact ($-1$)--surgery. The following observation is a strictly differential topological statement.

\begin{lemma}\label{lem:diffsur}
Let $X$ be a smooth 5--manifold and $M$ a codimension--2 submanifold. Consider a knot $K$ in $M$ and a framing $\tau$ of $K$ in $X$. Suppose that $\tau$ restricts to a framing $\tau_s$ of $K$ in $M$. Then a surgery on $X$ along $K$ with framing $\tau$ induces a surgery on $M$ along $K$ with framing $\tau_s$.
\end{lemma}

\begin{proof}
The statement can be seen as a consequence of the description of a surgery as a handle attachment. The gradient flow used to glue a 6--dimensional 2--handle $H^6\cong\D^2\times\D^4$ along the attaching sphere $K$ in $X\times\{1\}\subset X\times[0,1]$ restricts to a gradient flow in the submanifold $M\times\{1\}$. This describes the attachment of a 4--dimensional 2--handle $H^4\cong\D^2\times\D^2$ along $K$ in $M\times\{1\}\subset M\times[0,1]$. Note that the belt 3--sphere in the handle $H^6$ intersects the surgered submanifold $M_K$ along the belt 1--sphere of the handle $H^4$.
\end{proof}

Lemma \ref{lem:diffsur} provides the smooth model for the symplectic cobordism we shall construct to prove Proposition \ref{prop:symptrivnormal}. Proposition \ref{prop:dgs} concerns contact 3--manifolds and a 4--dimensional symplectic cobordism. In view of Lemma \ref{lem:diffsur} we can adapt Proposition \ref{prop:dgs} to the context of a codimension--2 contact submanifold in a contact 5--manifold. The result is as follows.

\begin{proposition}\label{prop:trivnorm} Let $(X,\xi_X)$ be a contact 5--manifold and $(M,\xi)$ a codimension--2 overtwisted contact submanifold. Consider a transverse knot $K$ in $M$ such that $c_1(\nu_M)=pd([\lambda_K])$ and denote $\lambda=\lambda_K\cup\lambda^\sigma_K$. Then there exists a framing $\tau$ of $\lambda$ in $(X,\xi_X)$ restricting to the Legendrian framing $\tau_s$ of $\lambda$ in $(M,\xi)$ such that $M_\lambda(\tau_s)$ is contactomorphic to $M$ with a Lutz untwist along $K$, and the symplectic normal bundle of $M_\lambda(\tau_s)$ in $X_\lambda(\tau)$ is trivial.
\end{proposition}

\begin{proof}
The contact ($-1$)--surgery that occurs on the contact 3--manifold $(M,\xi)$ is the procedure described in Proposition \ref{prop:dgs} and Corollary \ref{coro:dgs}. It suffices to explain the choice of framing $\tau$ for the link $\lambda$ in $X$. The Legendrian framing $\tau_s$ for $\lambda$ in $(M,\xi)$ is extended to a framing $\tau$ for $\lambda$ in $X$. This extension is obtained as follows.

Consider a section $\s:M\longrightarrow\nu_M$ transverse to the $0$--section and such that
$$\lambda_K=Z(\s),\mbox{ where }Z(\s)=\{p\in M:\s(p)=0\}.$$
This section exists since $c_1(\nu_M)=pd([\lambda_K])$. It is used to define the extension of the Legendrian framing $\tau_s$ to $\tau$.
Let us discuss in detail this and the effect of the surgery. It can be considered in two stages.

First, surgery along the Legendrian link $\lambda_K$. The required framing along $\lambda_K$ is defined to be $\tau=(\tau_s,\s_*\tau_s)$. Thus $\tau$ is constructed using the differential $\s_*$ of the section $\s$. The section $\s$ cannot be used since it vanishes along $\lambda_K$. Consider polar coordinates $(r,w_1,w_2)\in\D^4\subset\C^2$ with $(w_1,w_2)\in\S^3$. The framing $\tau$ provides a diffeomorphism
$$f_\tau:\S^1\times\D^2\times\D^2\longrightarrow \SU(\lambda_K)\subset X,\quad (\theta;r,w_1,w_2)\longmapsto f_\tau(\theta;r,w_1,w_2)$$
and we can suppose that $f_\tau(\S^1\times\D^2\times\{0\})=\SU(\lambda_K)\cap M$, for a neighborhood $\SU(\lambda_K)$ of $\lambda_K\subset X$. The differential $\s_*$ identifies the pull--back $f_\tau^*(\nu_M)$ of the normal bundle with the trivial bundle $\C\longrightarrow\S^1\times\D^2\times\{0\}$ over a neighborhood of $\lambda_K\subset M$. We can also suppose that the section $\s$ in these local coordinates is $((f^\tau)^*\s)(\theta;r,w_1)=rw_1$. The function $rw_1$ is well--defined although the coordinate $w_1$ is not well--defined at $r=0$.

The surgery substitutes the core $\lambda_K\cong\S^1\times\{0\}\times\{0\}\subset\S^1\times\D^2\times\D^2$ with coordinates $(\theta;r,w_1,w_2)$ by $\{0\}\times\S^3\subset\D^2\times\S^3$ with coordinates $(r,\theta;w_1,w_2)$ along the common boundary $\S^1\times\S^3=\{(\theta;w_1,w_2)\}$. The section $((f^\tau)^*\s)(\theta;r,w_1)=rw_1$ can be substituted by a section of the form
$$g:\D^2\times\S^3\longrightarrow\C,\quad (r,\theta;w_1,w_2)\longmapsto g(r,\theta;w_1,w_2)=\rho(r)w_1$$
where $\rho:\R\longrightarrow\R^+$ is a positive smooth function. In particular it is non--vanishing and provides a trivialization of the normal bundle of the surgered submanifold $M_{\lambda_K}(\tau_s)$ in the surgered manifold $X_{\lambda_K}(\tau)$.

Second, surgery along the Legendrian link $\lambda^\sigma_K$. The manifold $M_{\lambda_K}(\tau_s)$ has trivial normal bundle in $X_{\lambda_K}(\tau)$. Thus there exists a global framing $\tau_{\nu}$ of this normal bundle. Denote the restriction of this global framing $\tau_{\nu}$ to the Legendrian knot $\lambda^\sigma_K$ by $\tau_{\nu}|_{\lambda^\sigma_K}$. Then the framing $\{\tau_s, \tau_{\nu}|_{\lambda^\sigma_K}\}$ is a framing of the normal bundle of $\lambda^\sigma_K$ inside $X_{\lambda_K}(\tau)$. 
Thus, once the surgery along $\lambda^\sigma_K$ is performed with the framing $\{\tau_s, \tau|_{\lambda^\sigma_K} \}$, the resulting normal bundle is still trivial.
Hence the normal bundle of $M_\lambda(\tau_s)$ in $X_\lambda(\tau)$ is trivial.
\end{proof}

A minor modification of the argument for Proposition \ref{prop:trivnorm} yields the following result.

\begin{proposition}\label{prop:trivnormK} Let $(X,\xi_X)$ be a contact 5--manifold and $(M,\xi)$ a codimension--2 overtwisted contact submanifold with trivial normal bundle. Consider a transverse knot $K$ in $M$ and denote $\lambda=\lambda_K\cup\lambda^\sigma_K$. Then there exists a framing $\tau$ of $\lambda$ in $(X,\xi_X)$ restricting to the Legendrian framing $\tau_s$ of $\lambda$ in $(M,\xi)$ such that $M_\lambda(\tau_s)$ is contactomorphic to $M$ with a Lutz untwist along $K$ and the symplectic normal bundle of $M_\lambda(\tau_s)$ in $X_\lambda(\tau)$ is trivial.
\end{proposition}
\begin{proof}
In this case there is no need to use $\s_*$ since the section $\s$ can be chosen to be non--vanishing. Thus we choose the framing described in the second part of the surgery in Proposition \ref{prop:trivnorm}. Id est, the framing induced by $\s$. The surgery along $\lambda$ with this framing preserves the triviality of the normal bundle.
\end{proof}

We are now ready to prove Proposition \ref{prop:symptrivnormal}.

Let $(M,\xi)$ be an overtwisted contact 3--manifold. We want to show that there is a contact embedding with trivial symplectic normal bundle of $(M,\xi)$ into an exact symplectially fillable contact 5--manifold.

Fix an overtwisted disc $\Delta$ in $(M,\xi)$ and take a Legendrian link $L$ in $(M,\xi)$ which is disjoint from $\Delta$ and such that $pd([L]) = c_1(\xi)$. By Lemma \ref{lem:formLutz} we know that there exists a Legendrian link $\Lambda$ in $(M,\xi)$ disjoint from $L$ and $\Delta$ and such that $\xi$ is isotopic to $\overline{\xi} := \xi^{t(L \cup\Lambda)}$. Consider a contact embedding $(M,\overline{\xi}) \longrightarrow \S T^{\ast}M$ defined by some contact form $\overline{\alpha}$ for $\overline{\xi}$. The symplectic normal bundle of this embedding is isomorphic to $\overline{\xi}$ and hence to $\xi$. Note that $L$ and $\Lambda$ are still Legendrian in $(M,\overline{\xi})$. Consider the transverse push--offs, with respect to $\overline{\xi}$, $K=t(L)$ and $\kappa = t(\Lambda)$.

First, we apply Proposition \ref{prop:trivnorm} to $(M,\overline{\xi})$ inside $(X,\xi_X) := \S T^{\ast}M$, and $K=t(L)$. We can apply it because the symplectic normal bundle of $(M,\overline{\xi})$ inside $(X,\xi_X)$ is $\overline{\xi}$ and we know that
$$
c_1(\overline{\xi}) = c_1(\xi) = pd([L]) = pd([K]) = pd([\lambda_K]).
$$
The last equality holds by Lemma \ref{lem:h}. After applying Proposition \ref{prop:trivnorm} we get contact structures $\overline{\xi}'$ on $M$ and $\xi_X'$ on $X$ such that $(M,\overline{\xi}')$ embeds into $(X,\xi_X')$ with trivial symplectic normal bundle, and $\overline{\xi}'$, $\xi_X'$ are obtained from $\overline{\xi}$, $\xi_X$ by performing a Lutz untwist along $K$.

Second, consider $\kappa = t(\Lambda)$ as a transverse link in $(M,\overline{\xi}')$ and apply Proposition \ref{prop:trivnormK} to $(M,\overline{\xi}')$ inside $(X,\xi_X')$ and $\kappa = t(\Lambda)$. We obtain contact structures $\overline{\xi}''$ on $M$ and $\xi_X''$ on $X$ such that $(M,\overline{\xi}'')$ embeds into $(X,\xi_X'')$ with trivial symplectic normal bundle and $\overline{\xi}''$, $\xi_X''$ are obtained from $\overline{\xi}'$, $\xi_X'$ by performing a Lutz untwist along $\kappa$. 

Recall that $\overline{\xi}$ was obtained from $\xi$ by performing a Lutz twist along $K \cup \kappa$. We have thus that $\overline{\xi}''$ and $\xi$ are in the same homotopy class. Since the overtwisted disk has not been affected by the previous operations, Theorem \ref{thm:classification of ot} implies that the two contact structures $\overline{\xi}''$ and $\xi$ are actually isomorphic. We have thus obtained an embedding 
$$
(M,\xi) \longrightarrow (X,\xi_X'')
$$
with trivial symplectic normal bundle. Since $(X,\xi_X)$ is exact symplectically fillable and $\xi_X''$ is obtained from $\xi_X$ by two Lutz untwists, it follows from Corollary \ref{coro:dgs} and the discussion after it that $(X,\xi_X'')$ is still exact symplectically fillable. This finishes the proof of Proposition \ref{prop:symptrivnormal} and hence the proof of Theorem \ref{thm:otnotsmall} in the general case.

\end{document}